\documentclass[12pt,reqno,a4paper]{amsart}
\usepackage{extsizes}
\usepackage{blindtext}
\usepackage{fullpage}
\usepackage{mathtools}
\usepackage{amsmath,amssymb,amsthm}
\usepackage{amscd}
\usepackage{bm}
\usepackage{hyperref}
\usepackage{dsfont}
\usepackage{enumerate}
\usepackage{epsfig}
\usepackage{float, graphicx}
\usepackage{latexsym, amsxtra}
\usepackage{mathrsfs}
\usepackage{multicol}
\usepackage[normalem]{ulem}
\usepackage{psfrag}
\usepackage[parfill]{parskip}
\usepackage{stmaryrd}
\usepackage{tikz}
\usepackage[T1]{fontenc}
\usepackage{url}
\usepackage{verbatim}
\usepackage{indentfirst}
\usepackage{tikz-cd}

\usepackage{mathtools}

\makeatletter
\def\thm@space@setup{%
  \thm@preskip=2ex \thm@postskip=2ex
}
\makeatother

\hypersetup{hidelinks}

\newtheorem{thm}{Theorem~}[section]
\newtheorem{lem}[thm]{Lemma~}

\newtheorem{prop}[thm]{Proposition~}

\newtheorem{cor}[thm]{Corrolary~}

\newtheorem{rmk}[thm]{Remark~}

\newcommand{\CC}{\mathbb{C}}
\newcommand{\ZZ}{\mathbb{Z}}

\newcommand{\PP}{\mathbb{P}}
\newcommand{\N}{\mathcal{N}}
\newcommand{\B}{\mathcal{B}}
\newcommand{\M}{\mathcal{M}}
\newcommand{\Prd}{\mathscr{P}}
\newcommand\X{\mathscr{X}}

\newcommand\Aut{\mathrm{Aut}}
\newcommand\PGL{\mathrm{PGL}}
\newcommand\Hom{\mathrm{Hom}}
\newcommand\id{\mathrm{id}}

\newcommand\Ker{\mathrm{Ker}}
\newcommand\GIT{\mathrm{GIT}}
\newcommand\GL{\mathrm{GL}}

\newcommand\bs{\backslash}
\newcommand\dbs{\bs\!\!\bs}

\title{Orbifold Aspects of Certain Occult Period Maps}
\vspace{1.2cm}
 \author{Zhiwei Zheng}
 \date{\today}

 \newcommand{\Addresses}{{
  \bigskip
  \footnotesize

  Z.~Zheng, \textsc{Tsinghua University,
    Beijing, China}\par\nopagebreak
  \textit{E-mail address}: \texttt{zheng-zw14@mails.tsinghua.edu.cn}
}}

\begin{document}
\bibliographystyle{amsalpha}

\begin{abstract}
We first characterize the automorphism groups of Hodge structures of cubic threefolds and cubic fourfolds. Then we determine for some complex projective manifolds of small dimension (cubic surfaces, cubic threefolds, and non-hyperelliptic curves of genus $3$ or $4$), the action of their automorphism groups on Hodge structures of associated cyclic covers, and thus confirm conjectures made by Kudla and Rapoport in \cite{kudla2012occult}.
\end{abstract}

\maketitle

\section{Introduction}
Given a proper smooth family of K\"ahler manifolds, we can associate the polarized Hodge structure of each fiber to the base point, and hence obtain a holomorphic map from the base to the moduli space of polarized Hodge structures of certain fixed type. This holomorphic map is called the period map, which is a central notion in Hodge theory, and is a powerful tool for studying moduli spaces of projective manifolds for which the period map is injective (we then say these manifolds satisfy the global Torelli theorem).

\subsection{Occult Period Maps}

In \cite{kudla2012occult}, Kudla and Rapoport discussed what they called the occult period maps. The key point is that, for some kinds of projective manifolds, by looking at the periods of certain canonically associated objects instead of the usual periods, we obtain better characterization of the moduli spaces.  The examples addressed in \cite{kudla2012occult} are cubic surfaces, cubic threefolds, non-hyperelliptic curves of genus $3$ and $4$. We first sketch the constructions for those cases. More detailed treatments can be found in Section \ref{section: cubics} and Section \ref{section: kondo}.

\textbf{(Cubic surface)} For a smooth cubic surface $S$, we have $H^2(S, \CC)=H^{1,1}(S)$. Thus the Hodge structures on smooth cubic surfaces are without moduli. A clever construction by Allcock, Carlson and Toledo in \cite{allcock2002complex} is to consider the period of the cubic threefold $X$ which is a triple cover of $\PP^3$ branched along $S$. The celebrated work \cite{clemens1972intermediate} by Clemens and Griffiths showed the global Torelli for cubic threefolds. Therefore, the period of $X$ should control the geometry of $S$ in certain sense. The authors of \cite{allcock2002complex} associated the period of $X$ with $S$, and show that the resulting period map identifies the moduli space of smooth cubic surfaces with an open subset of an arithmetic ball quotient of dimension $4$. This period map is called the occult period map for cubic surfaces.

\textbf{(Cubic threefold)} For cubic threefolds, the usual period map gives rise to an embedding from the moduli space of smooth cubic threefolds to the moduli space of $5$ dimensional principal polarized abelian varieties. For this usual period map, the source has dimension $10$, while the target has dimension $15$. It turns out that an occult period map behaves better, in the sense that the source and target have the same dimension. To be concrete, let $T$ be a smooth cubic threefold. Denote by $X$  the triple cover of $\PP^4$ branched along $T$. Then $X$ is a cubic fourfold with a natural action by the group $\mu_3$ of third roots of unity.  The global Torelli theorem for cubic fourfolds is originally proved by Voisin \cite{voisin1986theoreme, voisin2008erratum}. A new and complete proof can also be found in \cite{looijenga2009period}.  In \cite{looijenga2007period} and \cite{allcock2011moduli}, the authors associated the period of $X$ with $T$, and show that the resulting period map identifies the moduli space of smooth cubic threefolds with an open subset of an arithmetic ball quotient of dimension $10$. This period map is called the occult period map for cubic threefolds.

\textbf{(Genus $3$ curve)} For a smooth non-hyperelliptic curve $C$ with genus $3$, the linear system of the canonical bundle $K_C$ embeds $C$ as a smooth quartic curve in $\PP^2$. Let $X$ be the fourth cover of $\PP^2$ branched along $C$. Then $X$ is a smooth quartic surface with a natural action by $\mu_4=\{\pm 1, \pm \sqrt{-1}\}$. A smooth quartic surface is a $K3$ surface of degree $4$. Global Torelli theorem for polarized $K3$ surfaces is first proved in \cite{pjateck1971torelli}. In \cite{kondo2000complex}, Kond\=o associated the period of $X$ with $C$, and showed that the resulting period map identifies the moduli space of smooth non-hyperelliptic curves of genus $3$ with an open subset of an arithmetic ball quotient of dimension $6$. This period map is called the occult period map for genus $3$ curves.

\textbf{(Genus $4$ curve)} For a smooth non-hyperelliptic curve $C$ with genus $4$, the linear system of the canonical bundle $K_C$ embeds $C$ as a complete intersection of a quadric surface $Q$ (either smooth or with one node) and a smooth cubic surface in $\PP^3$. Let $X$ be the triple cover of $Q$ branched along $C$. Then $X$ is a polarized $K3$ surface (either smooth or with one node) with a natural action by $\mu_3$. In \cite{kondo2000moduli}, Kond\=o associated the period of $X$ with $C$, and showed that the resulting period map identifies the moduli space of smooth non-hyperelliptic curves of genus $4$ with an open subset of an arithmetic ball quotient of dimension $9$. This period map is called the occult period map for genus $4$ curves.

The sources and targets of those four occult period maps acquire natural orbifold structures. In \cite{kudla2012occult}, Kudla and Rapoport regarded those four ball quotients as the coarse moduli of the moduli stack of abelian varieties with certain additional structures. Moreover, they reinterpreted the occult period maps as morphisms between Deligne-Mumford stacks. This led them to raise and partially answer some natural descent problems, for example, whether the occult period maps can be defined over their natural fields of definition.  See Section 9 in \cite{kudla2012occult}.

The main result of this paper, Theorem \ref{theorem: main} below, answers the conjectures made by Kudla and Rapoport about the orbifold aspects of the occult period maps, see Remark 5.2, 6.2, 7.2, 8.2 in \cite{kudla2012occult}. 
\begin{thm}[Main Theorem]
\label{theorem: main}
For smooth cubic surfaces, smooth cubic threefolds, smooth non-hyperelliptic curves with genus $3$ or $4$, the occult period maps identifies the orbifold structures on the moduli spaces and those on the ball quotients.
\end{thm}

\subsection{Structure of the Proof}

To prove Theorem \ref{theorem: main}, we need to characterize the actions of the automorphism groups of cubic threefolds, cubic fourfolds and polarized $K3$ surfaces on the corresponding polarized Hodge structures. The following fact is useful in this paper (see Proposition 2.11 in \cite{javanpeykar2017complete} combining with \cite{matsumura1963automorphisms}):

\begin{prop}
\label{proposition: faithful}
When $d\ge 3$, $n\ge 2$, and $X$ is a smooth degree $d$ $n$-fold, the induced action of $\Aut(X)$ on $H^n(X, \ZZ)$ is faithful.
\end{prop}

In order to prove Theorem \ref{theorem: main} for cubic threefolds, we need the following:
\begin{prop}
\label{proposition: main theorem cubic fourfold}
Let $X$ be a smooth cubic fourfold, then the group homomorphism 
\begin{equation}
\label{equation: arrow cubic fourfolds}
\Aut(X)\longrightarrow \Aut_{hs}(H^4(X, \ZZ), \eta)
\end{equation}
is an isomorphism. Here $\eta$ is the square of the hyperplane class of $X$, and $\Aut_{hs}(H^4(X, \ZZ), \eta)$ is the group of automorphisms of the lattice $H^4(X, \ZZ)$ preserving the Hodge decomposition and $\eta$.  
\end{prop}
The injectivity of the homomorphism \eqref{equation: arrow cubic fourfolds} is a corollary of Proposition \ref{proposition: faithful}. The surjectivity of the homomorphism \eqref{equation: arrow cubic fourfolds} is saying that any automorphism of the polarized Hodge structure on $H^4(X, \ZZ)$ is induced by an automorphism of $X$. We recall the global Torelli theorem for cubic fourfolds:
\begin{thm}[Voisin]
\label{theorem: voisin}
Let $X_1, X_2$ be two smooth cubic fourfolds. Suppose there exists an isomorphism $\varphi\colon H^4(X_2, \ZZ)\cong H^4(X_1, \ZZ)$ respecting the Hodge decompositions and squares of hyperplane classes, then there exists a linear isomorphism $f\colon X_1\cong X_2$.
\end{thm}

Actually, a stronger version of the global Torelli theorem for cubic fourfolds is claimed in \cite{voisin1986theoreme}. Namely, with the conditions in Theorem \ref{theorem: voisin}, the linear isomorphism $f\colon X_1\cong X_2$ can be uniquely chosen such that $\varphi$ is induced by $f$. Assuming the weak version (Theorem \ref{theorem: voisin}), the strong version of global Torelli is equivalent to Proposition \ref{proposition: main theorem cubic fourfold}. In Section \ref{section: automorphism of fourfold}, we show that Theorem \ref{theorem: voisin}, plus the injectivity of the group homomorphism \eqref{equation: arrow cubic fourfolds} appearing in Proposition \ref{proposition: main theorem cubic fourfold}, implies the surjectivity of the same homomorphism.  

\begin{rmk}
By \cite{beauville1985variety}, the Fano scheme of lines on a smooth cubic fourfold is a hyper-K\"ahler fourfold of deformation type $K3^{[2]}$. Via this construction, the strong version of global Torelli for cubic fourfolds can be deduced from Verbitsky's global Torelli theorem for hyper-K\"ahler manifolds. This is done by Charles \cite{charles2012remark}.
\end{rmk}

To show Theorem \ref{theorem: main} for cubic surfaces, we need to characterize the action of the automorphism group of a smooth cubic threefold on its intermediate Jacobian. Recall that for a smooth cubic threefold $X$, we denote $J(X)=H^3(X, \ZZ)\bs H^{1,2}(X)$, which is a $5$-dimensional complex torus with a principal polarization given by the topological intersection on $H^3(X, \ZZ)$. This principally polarized abelian variety $J(X)$ is called the intermediate Jacobian of $X$. See \cite{clemens1972intermediate}. By Proposition \ref{proposition: faithful}, we have an injective group homomorphism $\Aut(X)\hookrightarrow \Aut(J(X))$. Notice that we have naturally $\mu_2=\{\pm 1\}\subset \Aut(J(X))$. 

\begin{prop}
\label{proposition: main theorem cubic threefold}
Let $X$ be a smooth cubic threefold, then we have a natural group isomorphism $\Aut(J(X))\cong \Aut(X)\times \mu_2$.
\end{prop}
One input of our proof for Proposition \ref{proposition: main theorem cubic fourfold} and Proposition \ref{proposition: main theorem cubic threefold} is the existence of analytic slices for certain proper actions of complex Lie groups (see Proposition \ref{proposition: analytic slice theorem}), which implies the existence of universal deformations for any smooth hypersurfaces of degree at least $3$. We discuss this in Section \ref{section: universal deformation of smooth hypersurfaces}. As an application of the results in Section \ref{section: universal deformation of smooth hypersurfaces}, we construct the moduli spaces of marked hypersurfaces in Section \ref{section: moduli of marked hypersurface}. In Section \ref{section: automorphism of fourfold} and \ref{section: automorphism of threefold}, we present the proof of Proposition \ref{proposition: main theorem cubic fourfold} and Proposition \ref{proposition: main theorem cubic threefold} respectively. In Section \ref{section: cubics} we conclude Theorem \ref{theorem: main} for cubic surfaces and cubic threefolds.

The action of the automorphism group of a polarized $K3$ surface on the corresponding Hodge structure is well-understood, thanks to the work by Rapoport and Burns \cite{burns1975torelli}. In Section \ref{section: kondo}, we prove Theorem \ref{theorem: main} for smooth non-hyperelliptic curves with genus $3$ or $4$. Our proof relies on lattice theoretic analysis.

\textbf{Acknowledgement:} I thank my Ph.D advisor, Professor Eduard Looijenga, for guiding me to related papers and for his help along the way. I thank Professor Michael Rapoport, for helpful communication which pointed out the main problems. Thanks to Ariyan Javanpeykar, Radu Laza, Jialun Li, Gregory Pearlstein,  Chenglong Yu for helpful comments.

\section{Universal Deformation of Smooth Hypersurface}
\label{section: universal deformation of smooth hypersurfaces}
All algebraic varieties considered in this paper are over the complex field, and the topology we are using is the analytic topology. We use $\PP^n$ to denote the complex projective space of dimension $n$. For a complex vector space $V$ of finite dimension, we denote by $\PP V$ the projectivization of $V$. By a degree $d$ $n$-fold, we mean a hypersurface of degree $d$ in $\PP^{n+1}$. In this section we require $n\ge 2$, $d\ge 3$ and $(n,d)\ne (2,4)$.

Let $G$ be a complex Lie group acting on a complex manifold $M$. For $x\in M$, we denote by $Gx=\{gx\big{|}g\in G\}$ the orbit of $x$ and by $G_x=\{g\in G\big{|}gx=x\}$ the stabilizer group of $x$.

A subgroup $H$ of $G$ acts on $G\times M$ via $h(g,x)=(gh^{-1},hx)$ for $h\in H$ and $(g,x)\in G\times M$. We denote $G\times^H M= H\dbs (G\times M)$ if $H$ is finite.

Let $X$ be a degree $d$ $n$-fold. We denote by $\Aut (X)$ the group of automorphisms of $X$ induced from linear transformations of the ambient space. According to \cite{matsumura1963automorphisms} (Theorem 2), when $d\ge 3$, $n\ge 2$ and $(n,d)\ne (2,4)$, the group $\Aut (X)$ is equal to the usual automorphism group of $X$ consisting of regular automorphisms.  In particular, this is the case when $X$ is a smooth cubic of dimension $2, 3$ or $4$.

The vector space $Sym^d((\CC^{n+2})^*)$ consists of degree $d$ polynomials with $n+2$ variables. We denote by $\mathcal {C}^{n,d}\subset Sym^d((\CC^{n+2})^*)$ the subspace consisting of polynomials defining smooth degree $d$ $n$-folds. Recall that $\PP\mathcal{C}^{n,d}$ is the projectivization of $\mathcal{C}^{n,d}$. 

For $F\in \mathcal{C}^{n,d}$ and $g\in \GL(n+2,\mathbb{C})$, we define $g(F)=F\circ g^{-1}$. Thus we have a left action of $\GL(n+2, \CC)$ on $\mathcal{C}^{n,d}$. This induces a left action of $\PGL(n+2, \CC)$ on $\PP\mathcal{C}^{n,d}$. Take a point $x$ in $\PP\mathcal{C}^{n,d}$ and denote by $X$ the corresponding degree $d$ $n$-fold, we have $G_x=\Aut (X)$. In our cases, $G_x$ is finite, see \cite{matsumura1963automorphisms}, Theorem 1.

For a complex submanifold $S$ of $\PP\mathcal{C}^{n,d}$, we denote by $\mathscr{X}_S$ the tautological family of degree $d$ $n$-folds over $S$. The following result will be used in the proof of Proposition \ref{proposition: main theorem cubic fourfold} and Proposition \ref{proposition: main theorem cubic threefold}.

\begin{prop}
\label{proposition: universal deformation}
For a smooth degree $d$ $n$-fold $X$ with corresponding point $x\in \PP\mathcal{C}^{n,d}$, there exists a complex submanifold $S$ of $\PP\mathcal{C}^{n,d}$ containing $x$, which satisfies the following properties:
\begin{enumerate}[(i)]
\item For any point $x^{\prime}\in \PP\mathcal{C}^{n,d}$ with the corresponding hypersurface $X^{\prime}$ linearly isomorphic to $X$ via $f\colon X^{\prime}\longrightarrow X$, we can find an open neighborhood $U$ of $x^{\prime}$ in $\PP\mathcal{C}^{n,d}$, a map $U\longrightarrow S$, and a morphism $\widetilde{f}\colon \X_U\longrightarrow \X_S$, such that one have the following commutative diagram:
\begin{equation*}
\begin{tikzcd}
\X_U\arrow{r}{\widetilde{f}}\arrow{d} & \X_S \arrow{d}\\
U\arrow{r} & S
\end{tikzcd}
\end{equation*}
with $\widetilde{f}|_{\mathscr{X}_{x^{\prime}}}=f\colon X^{\prime}\longrightarrow X$. The choice of $\widetilde{f}$ is unique.

\item The submanifold $S$ is $G_x$-invariant. In other words, any automorphism $a$ of $X$ induces an automorphism $a\colon S\longrightarrow S$ of S. We denote by $\widetilde{a}\colon \X_S\longrightarrow \X_S$ the pullback of $a$ on $\X_S$. We then have the following commutative diagram:
\begin{equation*}
\begin{tikzcd}
\X_S \arrow{r}{\widetilde{a}}\arrow{d} & \X_S \arrow{d}\\
S \arrow{r}{a} & S
\end{tikzcd}
\end{equation*}

\item Suppose there are $x_1, x_2\in S$ and $g\in G$ with $g\colon\mathscr{X}_{x_1}\cong \mathscr{X}_{x_2}$, then $g\in G_x$.
\end{enumerate}
\end{prop}

To prove this theorem, we need to understand the local structure of the action of $\PGL(n+2, \CC)$ on $\PP\mathcal{C}^{n,d}$ at $x$. The following proposition should be known to the experts. However, we did not find it in the literature, hence we give a proof for completeness.

\begin{prop}
\label{proposition: analytic slice theorem}
Let $G$ be a complex Lie group acting holomorphically and properly on a complex manifold $M$. Suppose $x$ is a point in $M$ with the stabilizer group $G_x=\{g\in G\big{|} gx=x\}$ finite. Then there exists a smooth, locally closed, contractible, $G_x$-invariant submanifold $S$ of $M$ containing $x$, such that $GS$ is open and $G\times^{G_x}S\longrightarrow GS$ is an isomorphism. In particular, $G\times S\longrightarrow GS$ is a covering map of degree $|G_x|$.
\end{prop}

\begin{proof}
The orbit $Gx\cong G/G_x$ is a submanifold of $M$ containing $x$. There exists an open neighborhood $U$ of $x$ in $M$ with an open embedding $j\colon U\hookrightarrow T_x M$, such that $j(x)=0$, and the tangent map $j_*$ equals to identity. For every $g\in G_x$, the tangent map $g_*\colon T_x M\longrightarrow T_x M$ of $g$ at $x$ is an invertible linear map . Consider a holomorphic map $F\colon U\longrightarrow T_x M$ sending $y\in U$ to 
\begin{equation*}
F(y)=\frac{1}{|G_x|} \sum_{g\in G_x} (g_*^{-1} j(g(y))).
\end{equation*}
Then $F(x)=0$ and $F_*=\id$. Moreover, for any $h\in G_x$, we have 
\begin{equation}
\label{equation: linearization}
F(h(y))=\frac{1}{|G_x|} \sum_{g\in G_x} (g_*^{-1} j(gh(y)))=\frac{1}{|G_x|} \sum_{g\in G_x} h_*((gh)_*^{-1} j(gh(y)))=h_* F(y)
\end{equation}

The representation of $G_x$ on $T_x M$ has an invariant subspace $T_x (Gx)$. By representation theory of finite groups, there exists an invariant subspace $T_1$ such that $T_x(Gx)\oplus T_1=T_x M$. By inverse function theorem, we can choose an open neighborhood $U_1$ of $x$ in $U$, such that the restriction of $F$ on $U_1$ is an open embedding into $T_x M$. We may shrink $U_1$ such that $F(U)$ is the product of an open subset of $T_x(Gx)$ and a $G_x$-invariant open subset $B$ of $U_1$. By Equation \eqref{equation: linearization}, the submanifold $S=F^{-1}(B)$ of $M$ is $G_x$-invariant.

Consider the natural map $p\colon G\times S\longrightarrow M$. The tangent map of $p$ at $(1, x)$ is an isomorphism $p_*\colon T_1 G\oplus T_x S\cong T_x(Gx)\oplus T_1=T_x M$. Thus $p_*$ is an isomorphism at any points in certain neighborhood of $x$ in $G\times S$. If $p_*$ is an isomorphism at $(1, y)$ for $y\in S$, then $p_*$ is also an isomorphism at every point in $G\times \{y\}$. Actually, for any $g\in G$, we can consider the commutative diagram
\begin{equation*}
\begin{tikzcd}
G\times S \arrow{r}{p} & M\arrow{d}{g}\\
G\times S \arrow{r}{p}\arrow{u}{g^{-1}} & M
\end{tikzcd}
\end{equation*}
where the map in the first column is multiplying the first factor with $g^{-1}$ from the left. Thus we have $p=g\circ p\circ g^{-1}$. Taking derivatives at $(g, y)$, the above equation implies that $p_*$ is an isomorphism at $(g,y)$.

Thus we may shrink $S$ such that $p_*$ is an isomorphism at every point in $G\times S$. As a summary of above argument, there exists a $G_x$-invariant submanifold $S$ of $M$ containing $x$, such that $T_x S\oplus T_x(Gx)=T_x M$, and $p\colon G\times S\longrightarrow M$ is open. In particular $GS$ is an open subset of $M$.

The map $G\times S\longrightarrow GS$ is surjective and factors through $G\times^{G_x}S$. It suffices to show that we can suitably shrink $S$, such that $G\times^{G_x} S\longrightarrow GS$ is an isomorphism. We assume this can not be achieved and try to conclude contradiction.

We can find $(g, s), (g^{\prime}, s^{\prime})\in G\times S$, such that $gs=g^{\prime}s^{\prime}$, and $g^{-1}g^{\prime}\notin G_x$. Denote $g_1=g^{-1}g^{\prime}$, and $s_1=s^{\prime}$. Then we obtain a pair $(g_1, s_1)\in G\times S$, such that $g_1\notin G_x$ and $g_1 s_1\in S$.  We shrink $S$ to obtain $x\in S_2\subset S$, such that $S_2$ is a $G_x$-invariant open submanifold of $S$ and $s_1\notin S_2$. By our assumption, there exists $(g_2, s_2)\in G\times S_2$ such that $g_2\notin G_x$ and $g_2 s_2\in S_2$. 

Continuing to do this, we obtain a sequence of pairs $(g_i, s_i)_{i\in \mathbb{N}_+}$, such that $g_i\notin G_x$, $g_i s_i\in S_i\subset S$. We may require that the limit of $\overline{S_i}$ is the point $x$, then we have $s_i \rightarrow x$ as $i\rightarrow \infty$. The morphism $G\times M\longrightarrow M\times M$, $(g, x)\longmapsto (gx, x)$ is proper, hence the preimage of $\overline{S}\times \overline{S}\subset M\times M$ is compact. Thus there exists a subsequence $(g_{i_k}, s_{i_k})$ of $(g_i, s_i)$, such that $(g_{i_k}, s_{i_k})$ has a limit as $k\rightarrow\infty$. The limit of $(s_{i_k})$ must be $x$. Assume that $g_{i_k}\rightarrow g_0\in G$. Since $g_{i_k}s_{i_k}\in S_{i_k}$, we have $g_0 x=lim (g_{i_k} s_{i_k})$ equals to $x$. Thus $g_0\in G_x$. 

The differential of the morphism $G\times S\longrightarrow M$ at $(g_0, x)$ is an isomorphism $T_{g_0}G\oplus T_x S \cong T_{x}(Gx)\oplus T_x S\cong T_x M$. Therefore, $G\times S\longrightarrow M$ is a local isomorphism at $(g_0, x)$. This implies that $g_{i_k}=g_0$ for $k$ large enough. But by our choices we have $g_{i_k}\notin G_x$, which is a contradiction.
\end{proof}

In this paper, we call a submanifold $S$ with all the properties in Proposition \ref{proposition: analytic slice theorem} a slice for the action of $G$ on $M$ at $x$.

\begin{proof}[Proof of Proposition \ref{proposition: universal deformation}]
We consider the action of $G=\PGL(n+2,\mathbb{C})$ on $M=\PP\mathcal{C}^{n,d}$. By \cite{mumford1994geometric} (Proposition 0.8), this action is proper in the sense that $G\times M\longrightarrow M\times M$ is proper. By Proposition \ref{proposition: analytic slice theorem}, we can take $S$ to be a slice containing $x$. We next show that $S$ satisfies the properties we require.

($\romannumeral1$) Take $U$ to be an open neighborhood of $x^{\prime}$ in $GS$. Consider the covering map $G\times S\longrightarrow G\times^{G_x}S\cong GS$, we have a unique morphism $h\colon U\longrightarrow G\times S$ with $h(x^{\prime})=(f^{-1}, x)$, such that the following diagram commutes:
\begin{equation*}
\begin{tikzcd}
 &G\times S \arrow{d}{}\\
 U\arrow{ru}{h}\arrow{r}{}& G\times^{G_x}S
\end{tikzcd}
\end{equation*}
For $y^{\prime}\in U$, we denote $h(y^{\prime})=(g^{-1},y)$. Then we have $g^{-1}y=y^{\prime}$, hence $gy^{\prime}=y$. Thus, the lifting $h$ gives rise to a morphism $\widetilde{f}: \mathscr{X}_U\longrightarrow \mathscr{X}_{S}$ as required. The uniqueness of the lifting implies the uniqueness of $\widetilde{f}$.

$(\romannumeral 2)$ Recall that $G_x=\Aut(X)$. Since $S$ is $G_x$-invariant, the automorphism $a$ acts on $S$. The pullback $\widetilde{a}\colon \X_S\longrightarrow\X_S$ of $a$ satisfies the requirement.

$(\romannumeral 3)$ Consider the covering map $G\times S\longrightarrow G\times^{G_x}S\cong GS$. For any $h\in G_x$, the pair $(h,h^{-1}x_2)$ is a point in $G\times S$ over $x_2\in GS$. Since $g x_1=x_2$, the pair $(g,x_1)$ is also a point over $x_2$. Since $G\times S\longrightarrow GS$ is of degree $|G_x|$, one must have $(g,x_1)\in \{(h,h^{-1}x_2)\big{|}h\in G_x\}$, hence $g\in G_x$.
\end{proof}

\section{Moduli of Smooth Hypersurfaces with Markings}
\label{section: moduli of marked hypersurface}
In this section, all hypersurfaces are assumed to be smooth. We still assume that $n\ge 2$ and $d\ge 3$. We are going to construct the moduli space of marked degree $d$ $n$-folds as a complex manifold. 

Let be given a point $x\in M=\PP\mathcal{C}^{n,d}$ with $X=\mathscr{X}_x$ the corresponding degree $d$ $n$-fold. It is known that $H^n(X,\ZZ)$ is free. We have a unimodular bilinear form $b_x\colon H^n(X,\mathbb{Z})\times H^n(X,\mathbb{Z})\longrightarrow \mathbb{Z}$ given by the cup product. For $n$ even, we denote by $\eta_x\in H^n(X,\mathbb{Z})$ the $(n/2)$'th power of the hyperplane class. By a symmetric (symplectic) lattice, we mean a free abelian group of finite rank together with an integral symmetric (symplectic) bilinear form which is non-degenerate. Denote by $(\Lambda^{n,d}, b)$ an abstract lattice isomorphic to $(H^n(X,\mathbb{Z}),b_x)$. For $n$ even, we fix $\eta\in \Lambda^{n,d}$ such that $(\Lambda^{n,d}, b, \eta)\cong (H^n(X, \mathbb{Z}), b_x, \eta_x)$.

A marking of $X$ is an isomorphism $\phi\colon(H^n(X,\mathbb{Z}),b_x)\cong(\Lambda^{n,d},b)$ which sends $\eta_x$ to $\eta$ when $n$ is even. Two pairs $(x_1,\phi_1)$ and $(x_2,\phi_2)$ are said to be equivalent if there exists $g\in G=\PGL(n+2,\CC)$ such that $g(x_1)=x_2$ and $\phi_2=\phi_1\circ g^*$.

We define $\mathcal{N}^{n,d}$, the moduli space of marked smooth degree $d$ $n$-folds, firstly as a set, consisting of equivalence classes of $(x,\phi)$. We want to endow $\mathcal{N}^{n,d}$ with the structure of a complex manifold. We first identify the topology on $\mathcal{N}^{n,d}$.

Let be given $(x,\phi)\in \mathcal{N}^{n,d}$. We take $S$ to be a slice for the action of $G$ on $M$ at $x$. Recall that $G_x=\Aut(X)$ is the automorphism group of $X=\mathscr{X}_x$ and $\pi\colon \mathscr{X}_S\longrightarrow S$ is the tautological family of degree $d$ $n$-folds over $S$. Since $S$ is contractible, the local system $R^n \pi_*(\mathbb{Z})$ is trivializable. Thus $\phi$ induces a marking for every fiber of the local system. This gives rise to a map $q\colon S\longrightarrow \mathcal{N}^{n,d}$.

\begin{prop}
The map $q$ is injective.
\end{prop}

\begin{proof}
Suppose there are two different points $x_1, x_2\in S$ with $q(x_1)=q(x_2)$. We denote by $\phi_1,\phi_2$ the induced markings on $\X_{x_1},\X_{x_2}$. Then there exists a linear transformation $g\colon\mathscr{X}_{x_1}\longrightarrow \mathscr{X}_{x_2}$ with $\phi_2=\phi_1\circ g^*$.

We have $g\in G_x$ by Proposition \ref{proposition: universal deformation}. By Proposition \ref{proposition: faithful}, $g^*$ acts nontrivially on $H^n(X,\mathbb{Z})$. This implies that $\phi$ and $\phi\circ g^*$ are two different markings of $X$, hence $\phi_2$ and $\phi_1\circ g^*$ are two different markings of $\mathscr{X}_{x_2}$, a contradiction! We showed the injectivity of $q$.
\end{proof}

Now we take those slices as charts on $\mathcal{N}^{n,d}$. To make $\mathcal{N}^{n,d}$ a complex manifold, we still need to show it has the Hausdorff property.

\begin{prop}
With the topology given as above, $\mathcal{N}^{n,d}$ is Hausdorff.
\end{prop}
\begin{proof}
Suppose two pairs $(x_1,\phi_1)$, $(x_2,\phi_2)$, as points in $\mathcal{N}^{n,d}$, are non-separated. By \cite{matsusaka1964two} (theorem 2), the moduli space of degree $d$ $n$-folds, as a $\GIT$-quotient of $\PP\mathcal{C}^{n,d}$ by $\PGL(n+2,\mathbb{C})$, is separated. This implies that $\mathscr{X}_{x_1}$ and $\mathscr{X}_{x_2}$ are linearly isomorphic. Without loss of generality, we assume that $x_1=x_2$.

Taking a slice $S$ containing $x_1$. Since $(x_1,\phi_1), (x_1,\phi_2)\in \mathcal{N}^{n,d}$ are non-separated, there exsits two points $x_3, x_4\in S$, such that $(x_3,\phi_3),(x_4,\phi_4)$ represent the same point in $\N^{n,d}$ (here we write $\phi_3$ for the marking on $\mathscr{X}_{x_3}$ induced by $\phi_1$, and $\phi_4$ the marking on $\X_{x_4}$ induced by $\phi_2$). Then there exists $g\colon \mathscr{X}_{x_3}\cong \mathscr{X}_{x_4}$ with $\phi_4=\phi_3\circ g^*$. By Proposition \ref{proposition: universal deformation} we have $g\in G_x$. Then $\phi_2=\phi_1\circ g^*$ as markings on $\mathscr{X}_{x_1}$. Therefore, $(x_1,\phi_1)$ and $(x_1,\phi_2)$ represent the same point in $\mathcal{N}^{n,d}$. This implies that $\N^{n,d}$ is Hausdorff.
\end{proof}

\begin{cor}
The set $\mathcal{N}^{n,d}$, with local charts given as above, is a complex manifold.
\end{cor}

The space $\N^{n,d}$ may be disconnected. For a complete understanding, we recall some works by Beauville on monodromy group of the universal family of degree $d$ $n$-folds. Take a point $x\in \mathcal{C}^{n,d}$ and denote by $X=\X_x$ the corresponding smooth degree $d$ $n$-fold, there is a representation
\begin{equation*}
\label{modromy representation}
\rho\colon\pi_1(\mathcal{C}^{n,d}, x)\longrightarrow \Aut(H^n(X,\ZZ))
\end{equation*}
of the fundamental group $\pi_1(\mathcal{C}^{n,d},x)$ of $\mathcal{C}^{n,d}$. The image of $\rho$, denoted by $\Gamma_{n,d}$, is called the monodromy group of the universal family of smooth degree $d$ $n$-folds. From \cite{beauville1986groupe}, we have:
\begin{thm}[Beauville]
\label{lemma: monodromy group}
\begin{enumerate}[(i)]
\item For $n$ even, and $(n,d)\ne (2,3)$, we have $\Gamma_{n,d}\subset \Aut(H^n(X,\ZZ),b_x,\eta_x)$ of index two.
\item For $n=2$ and $d=3$, we have $\Gamma_{n,d}=\Aut(H^2(X,\ZZ),\eta_x)$ equals to the Weyl group of the $\mathrm{E}_6$ lattice.
\item For $n$ odd and $d$ even, we have $\Gamma_{n,d}=\Aut(H^n(X,\ZZ),b_x)$.
\item For $n$ odd and $d$ odd, there exists a quadratic form
      \begin{equation*}
      q_x\colon H^n(X,\ZZ)\longrightarrow \ZZ/2\ZZ
      \end{equation*}
      such that $q_x(u+v)=q_x(u)+q_x(v)+b_x(u,v)$ (for any $u,v\in H^n(X,\ZZ)$), and $\Gamma_{n,d}=\Aut(H^n(X,\ZZ),b_x,q_x)$.
   \end{enumerate}
\end{thm}

Since $\PP \mathcal{C}^{n,d}$ is connected, the connected components of $\mathcal{N}^{n,d}$ are in bijection with the cosets of the monodromy group in the target automorphism group. Thus we have the following:

\begin{cor}
The moduli space $\N^{n,d}$ of marked degree $d$ $n$-folds has finitely many connected components, precisely,
\begin{enumerate}[(i)]
\item It is connected if $(n,d)=(2,3)$, or $n$ odd and $d$ even,
\item it has two components if $n$ even and $(n,d)\ne (2,3)$,
\item for $n$ odd and $d$ odd, the number of its connected components is equal to $[\Aut(\Lambda, b):\Aut(\Lambda,b,q)]$, where $q$ is the $\ZZ/2\ZZ$-valued quadratic form on $\Lambda$ corresponding to $q_x$.
\end{enumerate}
\end{cor}

\section{Automorphism group of Cubic Fourfold}
\label{section: automorphism of fourfold}
In this section, we apply Proposition \ref{proposition: universal deformation} to investigate the relation between the automorphism group of a smooth cubic fourfold $X$ and that of the polarized Hodge structure of $X$. We will prove Proposition \ref{proposition: main theorem cubic fourfold}.

We first review some basic facts on Hodge theory of cubic fourfolds. Take $x\in \PP\mathcal{C}^{4,3}$ and denote by $X$ the corresponding cubic fourfold, then $H^4(X,\mathbb{Z})$ is a free abelian group of rank $23$, and the natural intersection pairing
\begin{equation*}
b_x\colon H^4(X,\mathbb{Z})\times H^4(X,\mathbb{Z})\longrightarrow \ZZ
\end{equation*}
is unimodular and of signature $(21,2)$. Recall from Section \ref{section: moduli of marked hypersurface} that we have $\eta_x\in H^4(X,\ZZ)$, and $(\Lambda^{4,3},b, \eta)\cong(H^4(X,\ZZ),b_x,\eta_x)$. 

Let $L$ be the orthogonal complement of $\eta$ in $\Lambda^{4,3}$, which is a lattice of signature $(20,2)$. Let $D$ be the projectivization of the set of points $x\in L_{\mathbb{C}}$ with $b(x,x)=0$ and $b(x,\overline{x})<0$. This is called the period domain of cubic fourfolds. The map $\mathscr{P}\colon \mathcal{N}^{4,3}\longrightarrow D$ taking $(x,\phi)$ to $\phi(H^{3,1}(\mathscr{X}_x))$ is the period map for cubic fourfolds.

\begin{prop}[Local Torelli Theorem for Cubic Fourfolds]
\label{proposition: local fourfold}
The period map $\mathscr{P}$ for cubic fourfolds is locally biholomorphic.
\end{prop}

\begin{proof}
The dimensions of $\mathcal{N}^{4,3}$ and $D_0$ are both equal to $20$. By Flenner's infinitesimal Torelli theorem (see \cite{flenner1986infinitesimal}, Theorem 3.1), the differential of $\mathscr{P}$ has full rank everywhere in $\mathcal{N}^{4,3}$. We conclude that $\Prd$ is locally biholomorphic.
\end{proof}

\begin{proof}[Proof of Proposition \ref{proposition: main theorem cubic fourfold}]
Take $x\in \PP\mathcal{C}^{4,3}$ and denote by $X$ the corresponding cubic fourfold. Denote by $\sigma$ an automorphism of $H^4(X,\mathbb{Z})$ which preserves $b_x$, $\eta_x$ and the Hodge structure.

Take a slice $S$ containing $x$. Take $\phi_1, \phi_2$ to be two markings of $X$, such that $\phi_2^{-1}\phi_1=\sigma$. For any $y\in S$, there are induced markings (from $\phi_1$,$\phi_2$) on $\mathscr{X}_y$, still denoted by $\phi_1$, $\phi_2$. Define two holomorphic maps $f_1, f_2$ from $S$ to $D$ by $f_i(y)=\mathscr{P}(y,\phi_i)$ for $i=1,2$.

By Proposition \ref{proposition: local fourfold}, we may assume $f_1, f_2$ to be open embeddings (shrink $S$ if necessary). Since $\sigma$ preserves Hodge structures, we have $f_1(x)=f_2(x)$. Then there exist two points $x_1,x_2$ in $S$, such that $f_1(x_1)=f_2(x_2)$ and this value in $D$ can be chosen generically. By Theorem \ref{theorem: voisin}, $\mathscr{X}_{x_1}$ and $\mathscr{X}_{x_2}$ are linearly isomorphic. We can choose a linear isomorphism $g\colon \mathscr{X}_{x_1}\cong \mathscr{X}_{x_2}$. By Proposition \ref{proposition: universal deformation} we have $g\in\Aut(X)$. Since $f_1(x_1)=f_2(x_2)$ is generic, it (as Hodge structures on $(L,b)$) admits no nontrivial automorphisms, hence $\phi_2=\phi_1\circ g^*$ as markings of $\mathscr{X}_{x_2}$. Then we have also $\phi_2=\phi_1\circ g^*$ as markings of $X$. Thus $\sigma=(g^{-1})^*$.
\end{proof}

\begin{cor}
The period map $\Prd\colon \mathcal{N}^{4,3}\longrightarrow D$ is an open embedding.
\end{cor}

\begin{proof}
Suppose $(x_1, \phi_1), (x_2, \phi_2)\in \N^{4,3}$ have the same image in $D$. Then by Theorem \ref{theorem: voisin}, there exists $g\in \PGL(6,\CC)$ with $g\colon \X_{x_1}\cong \X_{x_2}$. We have $(g^*)^{-1}\phi_1^{-1}\phi_2$ an automorphism of $H^4(\X_{x_2},\ZZ)$ preserving $b_{x_2}, \eta_{x_2}$ and the Hodge structure, hence it is induced by an automorphism of $\X_{x_2}$. This implies that $\phi_2^{-1}\phi_1$ is induced by a linear isomorphism between $\X_{x_1}$ and $\X_{x_2}$. Thus $(x_1,\phi_1)=(x_2,\phi_2)$ in $\N^{4,3}$. We showed the injectivity of $\Prd$, hence $\Prd$ is an open embedding.
\end{proof}
\section{Automorphism group of Cubic Threefold}
\label{section: automorphism of threefold}
In this section we deal with the case of cubic threefolds, and prove Proposition \ref{proposition: main theorem cubic threefold}. 

We first introduce the intermediate Jacobians of smooth cubic threefolds. Take $x\in \PP\mathcal{C}^{3,3}$ and denote by $X$ the corresponding cubic threefold, then $H^3(X,\mathbb{Z})$ is a free abelian group of rank $10$. There is a symplectic  unimodular bilinear form $b_x$ on $H^3(X, \mathbb{Z})$. The intermediate Jacobian of $X$ is defined to be $J(X)=H^{2,1}(X)\backslash H^3(X,\mathbb{C})/H^3(X,\mathbb{Z})$, which is a priori a complex torus. The symplectic form $b_x$ makes $J(X)$ a principally polarized abelian variety. We have the following theorem, see \cite{clemens1972intermediate} (Theorem 13.11), or \cite{beauville1982singularities}.

\begin{thm}[Global Torelli for Cubic Threefolds]
\label{theorem: global three}
Cubic threefolds are determined by their intermediate Jacobians. Precisely, if two cubic threefolds $X,Y$ have isomorphic intermediate Jacobians (as principal polarized abelian varieties), then they are isomorphic.
\end{thm}

We recall Griffiths' theory of integral of rational differentials on hypersurfaces, see \cite{griffiths1969periods}.

Take $F\in \mathcal{C}^{n,d}$ a degree $d$ polynomial of $n+2$ variables $Z_0,\dots,Z_{n+1}$, and denote by $Z(F)$ the zero locus of $F$ in $\PP^{n+1}$. We write 
\begin{equation*}
\Omega=\sum_{i=0}^{i=n+1} (-1)^i Z_idZ_0\wedge \dots \wedge\widehat{dZ_i}\wedge\dots\wedge dZ_{n+1}. 
\end{equation*}

Take an integer $a>0$ such that $ad\ge n+2$, and take a degree $ad-n-2$ polynomial $L$. We have a homogeneous rational differential $L\Omega/ F^a$ on $\mathbb{C}^{n+2}$, with its residue along $Z(F)$ giving rise to an $n$-form on $Z(F)$. Define $R\colon \mathbb{C}[Z_0,\dots,Z_{n+1}]_{ad-n-2}\longrightarrow H^n(Z(F),\mathbb{C})$ to be the map taking $L$ to $Res_{Z(F)}({L\Omega}/{F^a})$. We denote by
\begin{equation*}
F^n(Z(F))\subset\dots\subset F^0(Z(F))=H^n(Z(F),\CC)
\end{equation*}
the Hodge filtration on $H^n(Z(F),\CC)$. By \cite{griffiths1969periods}, we have:

\begin{thm}
\label{theorem: griffiths}
The map $R$ has image in $F^{n-a+1}(Z(F))$, and the composition of
\begin{equation*}
\CC[Z_0,\dots,Z_{n+1}]_{ad-n-2}\xrightarrow{R} F^{n-a+1}\to F^{n-a+1}/F^{n-a}\cong H^{n-a+1,a-1}(Z(F))
\end{equation*}
is surjective.
\end{thm}

\begin{lem}
\label{lemma: -id}
The automorphism $-\id$ of $J(X)$ is not induced by any automorphism of $X$.
\end{lem}

\begin{proof}
Suppose there is a linear isomorphism $g\colon X\longrightarrow X$ with $g^*=-\id$ on $J(X)$. Then $g^{*2}=\id$ on $H^3(X,\mathbb{Z})$. By Proposition \ref{proposition: faithful}, we have $g^2=\id$.

We can take a linear transformation $\widetilde{g}:\mathbb{C}^5\longrightarrow \mathbb{C}^5$ representing $g$, and choose a coordinate system $(Z_0,\dots, Z_4)$, such that $\widetilde{g}(Z_i)(=Z_i\circ \widetilde{g}^{-1})=Z_i$ or $-Z_i$ for $i\in\{0,1,\cdots, 4\}$. For each $i\in \{0,1,\cdots, 4\}$, there exists a complex number $\lambda_i$ with $\widetilde{g}({Z_i\Omega}/{F^2})=\lambda_i ({Z_i\Omega}/{F^2})$.

Since $g^*=-\id$, the automorphism $g$ is nontrivial, hence there exists $i_1, i_2\in \{0,1,\cdots, 4\}$ such that $\widetilde{g}(Z_{i_1})=Z_{i_1}$ and $\widetilde{g}(Z_{i_2})=-Z_{i_2}$.  Thus $\lambda_{i_1}\ne \lambda_{i_2}$.

On the other hand, by $g^*=-\id$ on $J(X)$ we have that $g^*=-\id$ on $H^3(X, \CC)$. By taking  residues of ${Z_i\Omega}/{F^2}$ along $X$, we obtain a basis for $H^{2,1}(X)$. Thus $\lambda_i=-1$ for every $i$. This contradicts the previous result $\lambda_{i_1}\ne \lambda_{i_2}$.
\end{proof}

Denote by $P$ the ambient space of $X$. For a linear form $l$ (of variables $Z_0, \dots, Z_4$), the rational differential $l\Omega/F^2$ has residue in $H^{2,1}(X)$.  Recall that $\PP H^{2,1}(X)$ is the projectivization of $H^{2,1}(X)$. We have a map $P^*\longrightarrow \PP H^{2,1}(X)$, where $P^*$ is the dual of $P$. By Theorem \ref{theorem: griffiths}, every element in $H^{2,1}(X)$ comes in this way. Thus the map $P^*\longrightarrow \PP H^{2,1}(X)$ is surjective. Since $\dim P^*= \dim P=\dim \PP H^{2,1}(X)=4$, we obtain an isomorphism $\kappa\colon P^{*}\cong \PP H^{2,1}(X)$. Notice that $\PP H^{2,1}(X)$ and $\PP H^{1,2}(X)$ are naturally dual to each other, we have an isomorphism $\kappa^{*-1}\colon P\cong \PP H^{1,2}(X)$.

\begin{lem}
\label{lemma: P}
For any $g\in \Aut(X)$, the following diagram commutes:
\begin{equation}
\label{diagram: P}
\begin{tikzcd}
P \arrow{d}{g}\arrow{r}{\kappa^{*-1}} &\PP H^{1,2}(X)\\
P \arrow{r}{\kappa^{*-1}}             &\PP H^{1,2}(X)\arrow{u}[swap]{g^*}
\end{tikzcd}
\end{equation}
\end{lem}
\begin{proof}
Let $\widetilde{g}\colon \CC^5\longrightarrow \CC^5$ be a linear isomorphism representing $g$. For an arbitrary linear form $l$, we have \begin{equation*}
\widetilde{g}^*(l\Omega_5/F^2)=\widetilde{g}^*(l)\widetilde{g}^*(\Omega_5)/(\widetilde{g}^*(F))^2=\lambda(g) \widetilde{g}^*(l)(\Omega_5 /F^2)
\end{equation*}
where $\lambda(g)$ is a complex number independent of $l$. This implies the commutativity of the following diagram:
\begin{equation}
\label{diagram: dual of P}
\begin{tikzcd}
P^* \arrow{d}{g^*}\arrow{r}{\kappa} &\PP H^{2,1}(X)\arrow{d}{g^*}\\
P^* \arrow{r}{\kappa}             &\PP H^{2,1}(X)
\end{tikzcd}
\end{equation}
which implies the commutativity of Diagram \eqref{diagram: P}.
\end{proof}

The theta divisor $\Theta$ of the intermediate Jacobian $J(X)$ has a unique singular point (using translation, we may ask the singular point to be $0$) of degree $3$, and the projectivized tangent cone $\PP T_0\Theta\subset \PP T_0 J(X)=\PP H^{1,2}(X)$ is identified with $X$ via $\kappa^*\colon \PP H^{1,2}(X)\cong P$, see \cite{beauville1982singularities} (main theorem) together with the discussion in \cite{clemens1972intermediate}, Chapter 12.

Take $\sigma\in \Aut(J(X))$ which induces a linear automorphism $\sigma_*$ of $\PP T_0 J(X)$. Since $\sigma$ preserves $\Theta$, it must fix the only singular point $0$. Thus the induced automorphism $\sigma_*$ preserves $X\subset P$. We obtain a group homomorphism
\begin{equation*}
\alpha\colon \Aut(J(X))\longrightarrow \Aut(X)
\end{equation*}
taking $\sigma$ to $\sigma_*^{-1}$.

An automorphism $g$ of $X$ induces $g^*\colon H^{1,2}(X)\longrightarrow H^{1,2}(X)$ preserving the lattice $H^3(X,\ZZ)\subset H^{1,2}(X)$. Thus $g^*$ gives rise to an automorphism of $J(X)$. In this way we obtain a group homomorphism
\begin{equation*}
\beta\colon \Aut(X)\longrightarrow \Aut(J(X)).
\end{equation*}
By Lemma \ref{lemma: P} we have $\alpha\beta=id$. Thus, $\Aut(J(X))\cong \Aut(X)\times \Ker(\alpha)$. 

\begin{proof}[Proof of Proposition \ref{proposition: main theorem cubic threefold}]
To prove Proposition \ref{proposition: main theorem cubic threefold}, it suffices to show $\Ker(\alpha)=\mu_2$.

Suppose we have $\sigma\in \Aut(J(X))$ such that $\sigma\ne id$ and $\alpha(\sigma)=\id$. Then $\sigma$ is acting trivially on $\PP  H^{1,2}(X)$, hence the action of $\sigma$ on $H^{1,2}(X)$ is by a scalar, denoted by $\zeta$. The action of $\sigma$ on $H^{2,1}$ is then by the scalar $\overline{\zeta}$. Any automorphisms of a polarized abelian variety must have finite order (see Proposition 8, Chapter VII, \cite{lang1959abelian}), hence $\sigma$ has finite order. We may then assume that $\zeta$ is an $n$'th root of unity. Since $H^3(X, \mathbb{Q})$ is a vector space over $\mathbb{Q}$, all primitive $n$'th roots of unity should appear as eigenvalues of the automorphism $\sigma$ on $H^3(X, \CC)$. But we know that only $\zeta$ and $\overline{\zeta}$ appear. Thus $n$ equals to $2$, $3$, $4$ or $6$. To show $\Ker(\alpha)=\mu_2$, it suffices to show that the cases $n=3,4,6$ do not appear.

Denote by $D$ the period domain associated to cubic threefolds. In other words, $D$ is the moduli space of Hodge structures on $\Lambda^{3,3}$ which have type weight $3$ and Hodge numbers $(0,5,5,0)$, and are principally polarized by $b$. Recall from Section \ref{section: moduli of marked hypersurface} that $\mathcal{N}^{3,3}$ is the moduli space of marked smooth cubic threefolds. We have the period map $\Prd\colon \N^{3,3}\longrightarrow D$.

An automorphism (with order $3,4$ or $6$) of $\Lambda^{3,3}$ with only eigenvalues $\zeta$ and $\overline{\zeta}$ uniquely determines a Hodge structure on $\Lambda^{3,3}$, hence a point in $D$. There are only countably many such automorphisms, determining countably many points in $D$. We denote by $I$ the subset of $D$ consisting of such Hodge structures.

Let $x\in \PP \mathcal{C}^{3,3}$ be the corresponding point of a smooth cubic threefold $X$. Assume there exists an automorphism $\sigma$ of $H^3(X,\mathbb{Z})$ which preserves $b_x$ and acts as scalar by $\zeta$ on $H^{1,2}(X)$, where $\zeta$ is equal to a primitive $3$'th, $4$'th, or $6$'th root of unity. We are going to derive contradiction.

Take a slice $S$ for the action of $\PGL(5, \CC)$ on $\PP \mathcal{C}^{3,3}$ at $x$. Let $\phi_1, \phi_2$ be two markings of $X$ such that $\phi_2^{-1}\phi_1=\sigma$. For any $y\in S$, there are induced markings (from $\phi_1$,$\phi_2$) on $\mathscr{X}_y$, still denoted by $\phi_1$, $\phi_2$. Define two holomorphic maps $f_1, f_2$ from $S$ to $D$ by $f_i(y)=\mathscr{P}(y,\phi_i)$ for $i=1,2$.

Since $\sigma$ preserves Hodge structures, we have $f_1(x)=f_2(x)$. By Flenner's infinitesimal Torelli theorem, we may assume $f_1, f_2$ to be injective on $S$ (after suitable shrinking of $S$). Since $\dim(f_1(S))=\dim(f_2(S))=10$ and $\dim(D)=15$, we have
\begin{equation*}
\dim(f_1(S)\cap f_2(S))\ge 5
\end{equation*}
Then there exist two points $x_1,x_2$ in $S$, such that $f_1(x_1)=f_2(x_2)$ and this value is not in $I$. Therefore, Proposition \ref{proposition: main theorem cubic threefold} holds for $\X_{x_1}$ and $\X_{x_2}$. By Theorem \ref{theorem: global three}, there exists a linear isomorphism $g\colon \mathscr{X}_{x_1}\cong \mathscr{X}_{x_2}$. The composition $g^{*}\phi_2^{-1}\phi_1$ is an automorphism of $H^3(\X_{x_1},\ZZ)$ preserving $b_{x_2}$ and the Hodge structure, hence lies in $\Aut(J(\X_{x_1}))\cong \Aut(\X_{x_1})\times \mu_2$. Without loss of generality, we can select $g$ such that $g^{*}\phi_2^{-1}\phi_1\in \mu_2$.  By Proposition \ref{proposition: universal deformation} we have $g\in G_x$. We have $g^{*}\phi_2^{-1}\phi_1=g^{*}\sigma\in \mu_2$ as automorphisms of $H^3(X, \ZZ)$, which implies that $(g^{-1})^*=\pm \sigma$. Then we have $g^{-1}=\alpha(\beta(g^{-1}))=\alpha((g^{-1})^*)=\alpha(\pm \sigma)=id$, which is impossible because $(g^{-1})^*=\pm \sigma$ is nontrivial.
\end{proof}

\section{Occult Period Map: Cubics}
\label{section: cubics}
In the remaining of this paper we will consider occult period maps for four cases successively, and finally confirm some conjectures made by Kudla and Rapoport in \cite{kudla2012occult}.

\subsection{Case of Cubic Surfaces}
In this subsection we deal with cubic surfaces. For details of the construction see \cite{allcock2002complex}.

Take $S$ to be a cubic surface, and $X$ the associated cubic threefold given as the triple cover of the projective space $\mathbb{P}^3$ branched along $S$. Then there is a natural action of the cyclic group of order $3$ on $X$ (Deck transformations of the ramified covering) and hence also on $H^3(X,\mathbb{Z})$ and the intermediate Jacobian $J(X)$ of $X$. Denote by $\sigma$ a generator of the group action.

Therefore, we have the group $\mu_6=\{\pm \id, \pm \sigma, \pm \sigma^2\}$ acting on $J(X)$. We denote by $A_0$ the subgroup of $A=\Aut(J(X))$ consisting of elements commuting with $\sigma$. Notice that $\mu_6$ lies in the center of $A_0$.

We can construct a group homomorphism from $\Aut(S)$ to $A_0/{\mu_6}$ as follows. Take $a\colon S\longrightarrow S$ to be an automorphism of $S$, we can lift it to an automorphism $\widetilde{a}$ of $X$, unique up to Deck transformations. The automorphism $\widetilde{a}$ of $X$ induces an automorphism of $J(X)$ which commutes with $\sigma$, hence also induces an element in $A_0/{\mu_6}$. This construction does not depend on the choices of the lifting of $a$.

The map attaching $J(X)$ (with the action of $\mu_6$) to the cubic surface $S$ is called the occult period map of cubic surfaces, which is an open embedding of the coarse moduli space $\PGL(4,\CC)\dbs \PP\mathcal{C}^{2,3}$ of smooth cubic surfaces into an arithmetic ball quotient $\Gamma\backslash\B^4$ of dimension $4$, where $\Gamma=\Aut(\Lambda^{3,3},\sigma)/\mu_6$, see \cite{allcock2002complex}. In \cite{kudla2012occult} (Remark 5.2) a conjecture about the stack aspect of the occult period map for cubic surfaces is made, which is already claimed as an implication of Theorem 2.20 in \cite{allcock2002complex}. We prove (Theorem \ref{theorem: cubic surface}) the conjecture in a more straightforward way.

\begin{prop}
\label{proposition: cubic surfaces}
The group homomorphism $\Aut(S)\longrightarrow A_0/{\mu_6}$ is an isomorphism.
\end{prop}

\begin{proof}
We first show the surjectivity. Let $\zeta\in A_0$ be an automorphism of $J(X)$ commuting with $\mu_6$. By {Proposition \ref{proposition: main theorem cubic threefold}}, one element in $\{\zeta, -\zeta\}$ is induced by an automorphism of the cubic threefold $X$. With the ambiguity of $\mu_6$ in mind, we may just assume that $\zeta$ is induced by an automorphism of $X$. We denote this automorphism by $\widetilde{a}$.

Since $\zeta=\widetilde{a}^*$ commutes with $\sigma$, by {Proposition \ref{proposition: faithful}} we have that $\widetilde{a}$ commutes with the Deck transformations of $X\longrightarrow \PP^3$. Therefore, $\widetilde{a}$ is induced by an automorphism $a$ of $S$. We showed the surjectivity.

Next we show the injectivity. Let $a$ be an automorphism of $S$ inducing the trivial element in the group $A/{\mu_6}$. Equivalently, there is a lifting $\widetilde{a}$ of $a$ such that $\widetilde{a}^*\in \mu_6$. We can compose $\widetilde{a}$ with Deck transformations, hence we can assume $\widetilde{a}^*\in \{\pm \id\}$. By {Lemma \ref{lemma: -id}}, we must have $\widetilde{a}^*=\id$ and by {Proposition \ref{proposition: faithful}}, $\widetilde{a}=\id$, hence $a=\id$. We showed the injectivity.
\end{proof}

\begin{thm}
\label{theorem: cubic surface}
The occult period map
\begin{equation*}
\Prd\colon \PGL(4,\CC)\dbs \PP\mathcal{C}^{2,3}\longrightarrow \Gamma\backslash\B^{4}
\end{equation*}
for smooth cubic surfaces identifies the orbifold structures of the $\GIT$-quotient $\PGL(4,\CC)\dbs \PP\mathcal{C}^{2,3}$ and the image in $\Gamma\backslash \B^4$.
\end{thm}
\begin{proof}
By \cite{allcock2002complex}, $\Prd$ is an isomorphism of analytic spaces onto its image, by {Proposition \ref{proposition: cubic surfaces}}, it identifies the natural orbifold structures on the source and image.
\end{proof}

\subsection{Case of Cubic Threefolds}
In this subsection we deal with cubic threefolds. For details of the construction see \cite{allcock2011moduli}.

Take $T$ to be a cubic threefold, and $X$ the associated cubic fourfold given as triple cover of the projective space $\PP^4$ branched along $T$. As in the case of cubic surfaces, one has an action $\sigma$ of order $3$ on the middle cohomology $H^4(X,\mathbb{Z})$ of $X$, which preserves the intersection pairing and square of the hyperplane class of $X$, and acts freely on the primitive part $H^4_0(X,\mathbb{Z})$. Therefore, we have the group $\mu_6=\{\pm id, \pm \sigma, \pm \sigma^2\}$ acting on the lattice $H^4_0(X,\mathbb{Z})$ (with intersection pairing of discriminant 3). We then denote by $A$ the subgroup of $\Aut(H^4_0(X,\mathbb{Z}))$ consisting of elements preserving Hodge structures, and $A_0$ the subgroup of $A$ consisting of elements commuting with $\sigma$. The center of $A_0$ contains $\mu_6$.

We can construct a group homomorphism from $\Aut(T)$ to $A_0/{\mu_6}$ as follows. Take $a: T\longrightarrow T$ to be an automorphism of $T$, we can lift it to $\widetilde{a}: X\longrightarrow X$, an automorphism of $X$, unique up to Deck transformations. The automorphism $\widetilde{a}$ of $X$ induces an automorphism of $H^4_0(X,\mathbb{Z})$ which commutes with $\sigma$, hence also induces an element in $A_0/{\mu_6}$ which does not depend on the choices of the lifting of $a$.

The map attaching Hodge structures on the lattice $H^4_0(X,\mathbb{Z})$ (preserved by the action of $\mu_6$) to the cubic threefolds $T$ is the occult period map for cubic threefolds, which is an open embedding of the coarse moduli space $\PGL(5,\CC)\dbs \PP\mathcal{C}^{3,3}$ of smooth cubic threefolds into an arithmetic ball quotient $\Gamma\backslash \B^{10}$, where $\Gamma=\Aut(\Lambda^{4,3}, \eta, \sigma)/\mu_3$ (see {Section \ref{section: moduli of marked hypersurface}} for the notations $\Lambda^{4,3},\eta$). We confirm the conjecture in \cite{kudla2012occult} (Remark 6.2) by the following proposition.

\begin{prop}
\label{proposition: threefolds}
The group homomorphism $\Aut(T)\longrightarrow A_0/{\mu_6}$ is an isomorphism.
\end{prop}

\begin{proof}
We first show the surjectivity. Let $\zeta\in A_0$ be an automorphism of $H^4_0(X,\mathbb{Z})$ preserving Hodge structure and commuting with $\sigma$. By lattice theory, one of $\zeta, -\zeta$ is induced by an automorphism of the whole cohomology $H^4(X,\mathbb{Z})$ which preserves square of the hyperplane section, hence by {Proposition \ref{proposition: main theorem cubic fourfold}} also induced by an automorphism of the cubic fourfold $X$. With the ambiguity of $\mu_6$ in mind, we may just assume that $\zeta$ is induced by an automorphism $\widetilde{a}$ of $X$.

Since $\zeta=\widetilde{a}^*$ commutes with $\sigma$, by {Proposition \ref{proposition: faithful}} we have that $\widetilde{a}$ commutes with the Deck transformations of $X\longrightarrow\PP^4$. Therefore, $\widetilde{a}$ is induced by an automorphism $a$ of $T$. We showed the surjectivity.

Next we show the injectivity. Let $a$ be an automorphism of $T$, inducing the trivial element in the group $A_0/{\mu_6}$. Equivalently, there is a lifting $\widetilde{a}$ of $a$ such that $\widetilde{a}^*\in \mu_6$. We can compose $\widetilde{a}$ with Deck transformations, hence we may assume that $\widetilde{a}^*\big{|}_{H^4_0(X,\mathbb{Z})}\in \{\pm \id\}$. Since $\widetilde{a}^*$ preserves square of the hyperplane class, we must have $\widetilde{a}^*=\id$. By {Proposition \ref{proposition: faithful}}, $\widetilde{a}=\id$, hence $a=\id$. We showed the injectivity.
\end{proof}

\begin{thm}
The occult period map
\begin{equation*}
\Prd_{3,3}\colon \PGL(5,\CC)\dbs \PP\mathcal{C}^{3,3}\longrightarrow \Gamma\backslash\B^{10}
\end{equation*}
for smooth cubic threefolds identifies the orbifold structures of the $\GIT$-quotient $\PGL(5,\CC)\dbs \PP\mathcal{C}^{3,3}$ and the image in $\Gamma\backslash\B^{10}$.
\end{thm}
\begin{proof}
By \cite{allcock2011moduli} (Theorem 1.9), $\Prd_{3,3}$ is an open embedding of analytic spaces. By {Proposition \ref{proposition: threefolds}}, it identifies the natural orbifold structures on the source and image.
\end{proof}

\section{Occult Period Map: Kond\=o's examples}
\label{section: kondo}
In this section we confirm Kudla and Rapoport's conjectures for non-hyperelliptic curves of genus $3$ and $4$. First we collect some results on $K3$ surfaces and lattice theory that will be used.

We will use the global Torelli theorem for $K3$ surfaces. The original literature is \cite{burns1975torelli}, one can also see \cite{huybrechts2016lectures}, \cite{looijenga1980/81torelli}.

\begin{thm}[Global Torelli Theorem for $K3$ Surfaces]
\label{theorem: global of K3}
Suppose two $K3$ surfaces $S_1$ and $S_2$ satisfy:
\begin{enumerate}[(i)]
\item There exists an isomorphism $\varphi\colon H^2(S_1,\mathbb{Z})\cong H^2(S_2,\mathbb{Z})$ preserving the corresponding Hodge structures,
\item $\varphi(\mathcal{K}_{S_1})\cap \mathcal{K}_{S_2}\ne \varnothing$, where $\mathcal{K}_{S_1}$ and $\mathcal{K}_{S_2}$ are the K\"ahler cones of $S_1$ and $S_2$.
\end{enumerate}
then there exists an isomorphism between the two $K3$ surfaces, and this isomorphism induces $\varphi$.
\end{thm}

Parallel to Proposition \ref{proposition: faithful}, one have the following lemma for $K3$ surfaces, see \cite{looijenga1980/81torelli} (Proposition 7.5).

\begin{lem}
\label{lemma: K3 faithful}
For any $K3$ surface $S$, the action of $\Aut(S)$ on $H^2(S,\mathbb{Z})$ is faithful.
\end{lem}

We recall some basic notions in lattice theory. One can refer to \cite{nikulin1980integral}. 

Let $M$ be a lattice. Denote $M_{\mathbb{Q}}=M\otimes\mathbb{Q}$ and still denote by $b_M$ the extended bilinear form on $M_{\mathbb{Q}}$. One has naturally $M\hookrightarrow \Hom(M,\mathbb{Z})\hookrightarrow M_{\mathbb{Q}}$. The lattice $M$ is called unimodular if $M\cong \Hom(M,\mathbb{Z})$.

The discriminant group of $M$ is defined to be $A_M=\Hom(M,\mathbb{Z})/M$. There is a quadratic form on $A_M$ defined as below:
\begin{align*}
q_M\colon A_M \longrightarrow \mathbb{Q}/{\mathbb{Z}}\\
[x]\longmapsto [b_M(x,x)]
\end{align*}
for $x\in \Hom(M,\ZZ)$ and $[x]\in A_M$ the equivalence class of $x$. This quadratic form $q_M$ is called the discriminant form associated with $M$.

If $b_M(x,x)\in 2\mathbb{Z}$ for any $x\in M$, then $M$ is called an even lattice. Suppose $M$ is even, then we can take values of the discriminant form $q_M$ in $\mathbb{Q}/(2\mathbb{Z})$. Suppose more that $M$ is $2$-elementary, i.e., $A_M$ is isomorphic to $(\mathbb{Z}/{2\mathbb{Z}})^l$ for certain integer $l$, then the image of $q_M$ lies in $(\frac{1}{2}\mathbb{Z})/(2\mathbb{Z})$. By \cite{nikulin1980integral} we have:
\begin{lem}
\label{lemma: glue automorphisms on lattices}
Suppose $L$ to be a unimodular lattice. Suppose $M, N$ to be two sublattices of $L$ perpendicular to each other (then both $M,N$ are primitive). Then the following hold.
\begin{enumerate}[(i)]
\item There is a natural isomorphism between $(A_M, q_M)$ and $(A_N, -q_N)$.
\item Suppose there are isomorphisms $\sigma_M\colon M\longrightarrow M$ and $\sigma_N\colon N\longrightarrow N$ inducing the same action on $A_M\cong A_N$, then there exists an automorphism of $L$ inducing $\sigma_M$ and $\sigma_N$.
\end{enumerate}
\end{lem}

\subsection{Case of Curves of genus $3$}
In this subsection we deal with curves of genus $3$. For details of the construction see \cite{kondo2000complex}.

Take $C$ to be a smooth non-hyperelliptic curve of genus $3$, which is embedded as a quartic curve in $\mathbb{P}^2$ by the canonical linear system. Take $S$ to be the associated quartic $K3$ surface given as degree $4$ cover of the projective space $\PP^2$ branched along $C$. There is a natural action of the cyclic group of order $4$ (Deck transformations of the ramified covering) on $S$, and hence also on $H^2(S,\mathbb{Z})$. Denote by $\sigma$ a generator of the order $4$ group.

Define $M=\{x\in H^2(S,\ZZ)\big{|}\sigma(x)=x\}$ and $N=\{x\in H^2(S,\ZZ)\big{|}\sigma(x)=-x\}$. They are primitive sublattices of $H^2(S,\ZZ)$, perpendicular to each other, and both have discriminant group isomorphic to $(\mathbb{Z}/{2\mathbb{Z}})^8$. The Hodge decomposition on $N$ restricted from that on $S$ has type $(1,14,1)$.

We have the group $\mu_4=\{\pm \id, \pm \sigma\}$ acting on the lattice $M$. We then denote by $A$ the subgroup of $\Aut(N)$ consisting of elements preserving the Hodge structure, and by $A_0$ the subgroup of $A$ consisting of elements commuting with $\sigma$.

We can construct a group homomorphism from $\Aut(C)$ to $A_0/{\mu_4}$ as follows. Take $a\colon C\longrightarrow C$ to be an automorphism of $C$ coming from a linear transformation of the ambient space $\mathbb{P}^2$. We can lift $a$ to an automorphism $\widetilde{a}$ of $S$, unique up to Deck transformations. The automorphism $\widetilde{a}$ of $S$ induces an automorphism of $N$ which commutes with $\sigma$, hence also induces an element in $A_0/{\mu_4}$ which does not depend on the choices of the lifting of $a$.

The map attaching the Hodge structure on $N$ (preserved by the action of $\mu_4$) to $C$ is the occult period map for smooth non-hyperelliptic curves of genus $3$, which is an open embedding of the coarse moduli space $\M^{\circ}_3$ of smooth non-hyperelliptic curves of genus $3$ into an arithmetic ball quotient $\Gamma\backslash \B^6$, where $\Gamma=\Aut(N,\sigma)/\mu_4$ is an arithmetic group acting on $\B^6$. We confirm the conjecture in \cite{kudla2012occult} (Remark 7.2) by the following proposition.

\begin{prop}
\label{proposition: genus 3}
The group homomorphism $\Aut(C)\longrightarrow A_0/{\mu_4}$ is an isomorphism.
\end{prop}

We need the following lemmas:

\begin{lem}
\label{lemma: genus 3 exercise}
For an $\mathrm{E}_7$-lattice $P$, we have a quadratic form $q\colon (\frac{1}{2}P)/P\longrightarrow \mathbb{Z}/(2\mathbb{Z})$ taking $x\in \frac{1}{2}P$ to $[2b_P(x,x)]$. Then we have an exact sequence:
\begin{equation*}
1\longrightarrow \{\pm \id\}\longrightarrow \Aut(P)\longrightarrow \Aut((\frac{1}{2}P)/P, q)\longrightarrow 1
\end{equation*}
\end{lem}

\begin{proof}
See \cite{bourbaki2002lie} (Exercise 3 of section 4, Chapter 6).
\end{proof}

\begin{lem}
\label{lemma: genus 3 aut surjective}
An automorphism of the lattice $N$ is induced by an automorphism of $H^2(S,\mathbb{Z})$ preserving the hyperplane class $\eta\in H^2(S,\mathbb{Z})$.
\end{lem}

This lemma is proved and used in \cite{kondo2000complex}. For completeness, we rewrite a proof.
\begin{proof}[Proof of {Lemma \ref{lemma: genus 3 aut surjective}}]
Let $D$ be the double cover of $\mathbb{P}^2$ branched along the quartic curve $C$, then $D$ is a Del Pezzo surface of degree $2$ and $S$ is a double cover of $D$ branched along $C$. The middle cohomology $H^2(D,\ZZ)$ of $D$ is a unimodular lattice, and $M\cong H^2(D,\ZZ)(2)$. Here we use $L(n)$ to denote a lattice $L$ with a scaled quadratic form by $n$. We have the discriminant group $A_M=(\frac{1}{2}M)/M$ of $M$. We have a sublattice $(\eta_0)\oplus P$ in $H^2(D,\ZZ)$ of index $2$, where $\eta_0$ is the hyperplane class of $D$ and $P$ is an $\mathrm{E}_7$-lattice. 

Denote by $\zeta$ an automorphism of $N$, it induces an automorphism of $(A_N, q_N)\cong (A_M, -q_M)$. It suffices to construct an automorphism $\rho$ of $M$, such that $\rho(\eta)=\eta$ and $\rho$, $\zeta$ induces the same automorphism of $(A_M, q_M)$.

The finite group $(\frac{1}{2}P)/P$ is a subgroup of $A_M\cong (\frac{1}{2}M)/M$. We are going to show that the induced map of $\zeta$ on $A_M$ preserves $(\frac{1}{2}P)/P$.

Take an element $x\in P$, consider $[\frac{1}{2}x]\in A_M$, then
\begin{equation*}
q_M([\frac{1}{2}x])=[\frac{1}{4}b_M(x,x)]=[\frac{1}{2}b_P(x,x)]\in \mathbb{Z}/(2\mathbb{Z})
\end{equation*}
where the last step is because $P$ is an $\mathrm{E}_7$-lattice, which is an even lattice. Since $H^2(D,\mathbb{Z})$ is an odd lattice, there exists element $y\in H^2(D, \mathbb{Z})$ with self-intersection an odd number, hence $q_M([\frac{1}{2}y])\notin \mathbb{Z}/(2\mathbb{Z})$. Therefore, as a subset of $A_M$, $(\frac{1}{2}P)/P=\{\alpha\in A_M\big{|}q_M(\alpha)\in \mathbb{Z}/(2\mathbb{Z})\}$, which implies that $\zeta$ preserves $(\frac{1}{2}P)/P$.

By {Lemma \ref{lemma: genus 3 exercise}}, there are two automorphisms $\rho_1$, $-\rho_1$ of $P$, both inducing the action $\zeta$ on $(\frac{1}{2}P)/P$. We can extend the action $\id\oplus \rho_1$ on $(\eta_0)\oplus P$ uniquely to an automorphism $\rho_2$ of $H^2(D,\mathbb{Z})$, and similarly extend $\id \oplus (-\rho_1)$ to $\rho_3$. These two automorphisms $\rho_2$ and $\rho_3$ can be regarded as automorphisms of $M$, hence also induce actions on $A_M$. Consider the automorphisms $\xi_1=\rho_2^{-1} \circ \zeta$ and $\xi_2=\rho_3^{-1}\circ \zeta$ on $(A_M, q_M)$, they are different and both act as identity on $(\frac{1}{2}P)/P$. 

Assume $\xi\colon A_M\longrightarrow A_M$ is an automorphism preserving $q_M$ and acting trivially on $(\frac{1}{2}P)/P$. Take $x\in M$ with $[\frac{1}{2}x]\notin (\frac{1}{2}P)/P$, assume $\xi([\frac{1}{2}x])=[\frac{1}{2}y]$ for $y\in M$. Then for any $z\in P$, we have $\xi([\frac{x+z}{2}])=[\frac{y+z}{2}]$, which implies that $q_M([\frac{x+z}{2}])=q_M([\frac{y+z}{2}])$. Thus $\frac{1}{2}(b_M(x-y, z))\in 2\ZZ$ for any $z\in P$. This implies that either $x-y$ or $x-y-\eta$ belongs to $2M$, hence $\xi([\frac{1}{2} x])=[\frac{1}{2} x]$ or $[\frac{1}{2} (x-\eta)]$. Therefore, the automorphism $\xi$ as required has at most two possibilities. We conclude that either $\xi_1$ or $\xi_2$ equals to identity, hence either $\rho_2$ or $\rho_3$ equals to $\zeta$ as automorphisms of $A_M$.
\end{proof}

\begin{lem}
\label{lemma: genus 3 aut uniqueness}
Suppose there are two automorphisms $\zeta_1$, $\zeta_2$ of the $K3$ lattice $H^2(S,\mathbb{Z})$ such that
\begin{equation*}
\zeta_1\big{|}_N=\zeta_2\big{|}_N\colon N\longrightarrow N
\end{equation*}
and both the two automorphisms preserve the hyperplane class, then they coincide.
\end{lem}
\begin{proof}
It suffices to show that, any automorphism $\zeta$ of $H^2(S, \mathbb{Z})$ which acts identically on $(\eta)\oplus N$ must be the identity.

Define sublattice $P$ of $H^2(D, \mathbb{Z})$ as in the proof of {Lemma \ref{lemma: genus 3 aut surjective}}. Since $\zeta$ acts identically on $N$, it also acts identically on $A_N\cong A_M$, hence also identically on $(\frac{1}{2}P)/P$. By {Lemma \ref{lemma: genus 3 exercise}} we have $\zeta$ equals to $\id$ or $-\id$ on $P$, with the latter possibility excluded by the fact that $\zeta$ is an automorphism of the whole lattice $H^2(S,\mathbb{Z})$ preserving $\eta$. Thus $\zeta=\id$ and we proved the lemma.
\end{proof}

\begin{proof}[Proof of {Proposition \ref{proposition: genus 3}}]
We first show the surjectivity. Let $\zeta\in A_0$ be an automorphism of $N$ preserving the Hodge structure and commuting with $\sigma$. By {Lemma \ref{lemma: genus 3 aut surjective}}, $\zeta$ is induced by an automorphism of the whole lattice $H^2(S,\mathbb{Z})$ which preserves the hyperplane class. This automorphism apparently preserves the Hodge structure on $H^2(S,\mathbb{Z})$, hence comes from an automorphism $\widetilde{a}$ of the quartic surface $S$.

Since $\zeta=\widetilde{a}^*\big{|}_N$ commutes with $\sigma$, we have $\sigma \widetilde{a}^*$ and $\widetilde{a}^* \sigma$ coincide on the lattice $N$ and both preserve the hyperplane class. By {Lemma \ref{lemma: genus 3 aut uniqueness}}, the equality $\sigma \widetilde{a}^*=\widetilde{a}^* \sigma$ holds on the whole lattice $H^2(S,\mathbb{Z})$. By {Lemma \ref{lemma: K3 faithful}} we have that $\widetilde{a}$ commutes with the Deck transformations of $S\longrightarrow\mathbb{P}^2$. Therefore, $\widetilde{a}$ is induced by an automorphism $a$ of $C$. We showed the surjectivity.

Next we show the injectivity. Let $a$ be an automorphism of $C$ inducing the trivial element in the group $A_0/{\mu_4}$. Then there is a lifting $\widetilde{a}$ of $a$ such that $\widetilde{a}^*\big{|}_N\in \mu_4$. We can compose $\widetilde{a}$ with Deck transformations, hence we can assume that $\widetilde{a}^*\big{|}_N=\id$. Since $\widetilde{a}^*$ acts as identity on the hyperplane class of $S$, by {Lemma \ref{lemma: genus 3 aut uniqueness}}, $\widetilde{a}^*=\id$ and by {Lemma \ref{lemma: K3 faithful}}, $\widetilde{a}=\id$, hence $a=\id$. We showed the injectivity.
\end{proof}

\begin{thm}
The occult period map
\begin{equation*}
\Prd\colon \M^{\circ}_3\longrightarrow \Gamma\backslash \B^6
\end{equation*}
for smooth non-hyperelliptic curves of genus $3$ identifies the natural orbifold structure of $\M^{\circ}_3$ and the image in $\Gamma\backslash \B^6$.
\end{thm}
\begin{proof}
By \cite{kondo2000complex} (Theorem 2.5) $\Prd$ is an open embedding of analytic spaces, combining with {Proposition \ref{proposition: genus 3}}, we have that $\Prd$ identifies the orbifold structures on the source and image.
\end{proof}

\subsection{Case of Curves of genus 4}
In this subsection we deal with curves of genus 4. For details of the construction see \cite{kondo2000moduli}.

Take $C$ to be a smooth non-hyperelliptic curve of genus $4$, which is embedded as a complete intersection of a quadric surface $Q$ (smooth or with one node) and a smooth cubic surface in $\PP^3$ via the canonical linear system. Take $S$ to be the associated $K3$ surface given as triple cover of the quadric surface $Q$ branched along $C$ (In case $Q$ is singular, take its minimal resolution instead). Then there is a natural action of the cyclic group of order $3$ on $S$ (Deck transformations of the ramified covering) and hence also on $H^2(S,\mathbb{Z})$. Denote by $\sigma$  a generator of this group.

Suppose the quadric surface containing $C$ is smooth, then it is isomorphic to $\PP^1\times \PP^1$; if the quadric surface is singular, then we can blow up the singular point and obtain $Q$ a rational surface which is the projectivization of the degree $2$ and rank $2$ vector bundle on $\PP^1$. In both the two cases we have $U=H^2(Q,\mathbb{Z})$ a hyperbolic lattice with generators $x_1, x_2$ such that $b_U(x_1,x_1)=b_U(x_2,x_2)=0, b(x_1, x_2)=1$ and $\eta_0=x_1+x_2$ is the hyperplane class of $Q$.

Denote $M=\{x\in H^2(S,\mathbb{Z})\big{|}\sigma(x)=x\}$ and $N=M^{\perp}$. Then $M$ contains the hyperplane class. Moreover, $M,N$ are primitive sublattices of $H^2(S,\mathbb{Z})$ perpendicular to each other. Explicitly, $M\cong H^2(Q,\mathbb{Z})(3)$ is of rank $2$, $N$ is of rank $20$, and they have isomorphic discriminant groups $A_N\cong A_M\cong (\mathbb{Z}/{3\mathbb{Z}})^2$. The induced Hodge docomposition on $N$ is of type $(1,18,1)$.

We have the group $\mu_6=\{\pm \id, \pm \sigma, \pm\sigma^2\}$ acting on the lattice $N$. We then denote by $A$ the subgroup of $\Aut(N)$ consisting of elements preserving the Hodge structure, and by $A_0$ the subgroup of $A$ consisting of elements commuting with $\sigma$.

We can construct a group homomorphism from $\Aut(C)$ to $A_0/{\mu_6}$ as follows. Take $a\colon C\longrightarrow C$ to be an automorphism of $C$ coming from a linear transformation of the ambient space $\mathbb{P}^3$. This linear transformation preserves $Q$ and we can lift it to $\widetilde{a}\colon S\longrightarrow S$, an automorphism of $S$, unique up to Deck transformations. The automorphism $\widetilde{a}$ of $S$ induces an automorphism of $N$ which commutes with $\sigma$, hence also induces an element in $A_0/{\mu_6}$ which does not depend on the choices of the lifting of $a$.

The map attaching the Hodge structure on $N$ (preserved by the action of $\mu_6$) to $C$ is the occult period map for smooth non-hyperelliptic curves of genus $4$, which is an open embedding of the coarse moduli space $\M^{\circ}_4$ of smooth non-hyperelliptic curves of genus $4$ into an arithmetic ball quotient $\Gamma\backslash\B^9$, where $\Gamma=\Aut(N,\sigma)/\mu_6$ is an arithmetic group acting on $\B^9$. We confirm the conjecture in \cite{kudla2012occult} (Remark 8.2) by the following proposition.

\begin{prop}
\label{occult_theorem_genus4}
The group homomorphism $\Aut(C)\longrightarrow A_0/{\mu_6}$ is an isomorphism.
\end{prop}

We need the following lemmas:

\begin{lem}
\label{lemma: genus 4 hyperbolic lattice}
Let $U$ be a hyperbolic lattice, i.e., with generators $x_1, x_2$ such that $b_U(x_1, x_1)=b_U(x_2,x_2)=0$, $b_U(x_1, x_2)=1$. Then all possible automorphisms $\rho$ of $U$ are in the list below:
\begin{enumerate}[(i)]
\item $\rho=\pm \id$,
\item $\rho(x_1)=x_2, \rho(x_2)=x_1$,
\item $\rho(x_1)=-x_2, \rho(x_2)=-x_1$.

\end{enumerate}
\end{lem}

\begin{proof}
The proof of this lemma is straightforward.
\end{proof}

\begin{lem}
\label{lemma: genus 4 surj}
Suppose $\zeta$ to be an automorphism of the lattice $N$, then exact one of $\pm\zeta$ is induced by an automorphism of $H^2(S,\mathbb{Z})$ preserving the hyperplane class $\eta\in H^2(S,\mathbb{Z})$.
\end{lem}

\begin{proof}
The automorphism $\zeta$ of $N$ induces an action on
\begin{equation*}
A_N\cong A_M=(\frac{1}{3}U)/U=\{0, \pm[\frac{1}{3}x_1], \pm[\frac{1}{3}x_2], \pm[\frac{1}{3}(x_1+x_2)], \pm[\frac{1}{3}(x_1-x_2)]\}
\end{equation*}

Exactly one of $\pm\zeta$ preserves $[\frac{1}{3}\eta_0]=[\frac{1}{3}(x_1+x_2)]$. Without loss of generality, we assume that $\zeta$ satisfies this property. Then $\zeta$ must send $[\frac{1}{3}x_1]$ to $[\frac{1}{3}x_1]$ or $[\frac{1}{3}x_2]$, and the value $\zeta([\frac{1}{3}x_2])$ is correspondingly determined. Combining with {Lemma \ref{lemma: genus 4 hyperbolic lattice}}, there exists an automorphism of $M=U(3)$ which preserves $\eta$ and matches with $\zeta$ on $N$. Thus by {Lemma \ref{lemma: glue automorphisms on lattices}}, the automorphism $\zeta$ is induced from an automorphism of the whole lattice $H^2(S,\mathbb{Z})$ which preserves $\eta$. This proves our lemma.
\end{proof}

\begin{lem}
\label{lemma: genus 4 inj}
Suppose there are two automorphism $\zeta_1$, $\zeta_2$ of the $K3$ lattice $H^2(S,\mathbb{Z})$ such that $\zeta_1\big{|}_N=\zeta_2\big{|}_N\colon N\longrightarrow N$. Then they coincide.
\end{lem}

\begin{proof}
Since $\zeta_1, \zeta_2$ act the same on $N$, they also act the same on $A_N\cong A_M$. By Lemma \ref{lemma: genus 4 hyperbolic lattice}, we know that $\zeta_1, \zeta_2$ act the same on $M$, hence the same on the whole lattice $H^2(S,\mathbb{Z})$.
\end{proof}

\begin{proof}[Proof of {Proposition \ref{occult_theorem_genus4}}]
We first show the surjectivity. Let $\zeta\in A_0$ be an automorphism of $N$ commuting with $\sigma$ and preserving the Hodge structure. By {Lemma \ref{lemma: genus 4 surj}}, one element in $\{\zeta, -\zeta\}$ is induced by an automorphism of the whole lattice $H^2(S,\mathbb{Z})$ which preserves Hodge structure and $\eta$. By {Theorem \ref{theorem: global of K3}}, this automorphism is induced by an automorphism of $S$. With the ambiguity of $\mu_6$ in mind, we may just assume that $\zeta$ is induced by an automorphism $\widetilde{a}$ of $S$.

Since $\zeta=\widetilde{a}^*\big{|}_N$ commutes with $\sigma$, By {Lemma \ref{lemma: genus 4 inj}} we have $\sigma \widetilde{a}^*=\widetilde{a}^* \sigma$ on $H^2(S,\ZZ)$. By {Lemma \ref{lemma: K3 faithful}} we have that $\widetilde{a}$ commutes with the Deck transformations of $S\longrightarrow Q$. Therefore, $\widetilde{a}$ is induced by an automorphism $a$ of $C$. We showed the surjectivity.

Next we show the injectivity. Let $a$ be an automorphism of $C$, inducing the trivial element in the group $A_0/{\mu_6}$. Then there is a lifting $\widetilde{a}$ of $a$ such that $\widetilde{a}^*\big{|}_N\in \mu_6$. We can compose $\widetilde{a}$ with Deck transformations, hence we can assume $\widetilde{a}^*\big{|}_N\in \{\pm \id\}$. Since $\widetilde{a}^*$ acts as identity on the hyperplane class of $S$, we must have $\widetilde{a^*}\big{|}_N=\id$ and by {Lemma \ref{lemma: genus 4 inj}}, $\widetilde{a}^*=\id$. Thus by {Lemma \ref{lemma: K3 faithful}}, $\widetilde{a}=\id$, which implies that $a=\id$. We showed the injectivity.
\end{proof}

\begin{thm}
The occult period map
\begin{equation*}
\Prd\colon \M^{\circ}_4\longrightarrow \Gamma\backslash \B^9
\end{equation*}
for smooth non-hyperelliptic curves of genus $4$ identifies the natural orbifold structures on $\M^{\circ}_4$ and the image in $\Gamma\backslash \B^9$.
\end{thm}
\begin{proof}
By \cite{kondo2000moduli}, $\Prd$ is an open embedding of analytic spaces, combining with {Proposition \ref{occult_theorem_genus4}} we have that $\Prd$ identifies the orbifold structures on the source and image.
\end{proof}

\bibliography{reference}
\bibliographystyle{alpha}

\Addresses
\end{document}